\documentclass{amsart}

\usepackage[all]{xypic}
\usepackage{url}

\newtheorem{theorem}{Theorem}[section]

\newtheorem{proposition}[theorem]{Proposition}
\newtheorem{corollary}[theorem]{Corollary}

\theoremstyle{definition}
\newtheorem{definition}[theorem]{Definition}
\newtheorem{example}[theorem]{Example}

\theoremstyle{remark}
\newtheorem{remark}[theorem]{Remark}

\numberwithin{equation}{section}

\newcommand{\gen}[1]{\langle#1\rangle}

\newcommand{\NN}{\mathbb{N}}
\newcommand{\CC}{\mathbb{C}}
\newcommand{\PP}{\mathbb{P}}
\newcommand{\TT}{\mathbb{T}}
\newcommand{\Tf}{\mathcal{T}}
\newcommand{\Nf}{\mathcal{N}}
\newcommand{\Ho}{\mathcal{H}om}
\newcommand{\OO}{\mathcal{O}}

\DeclareMathOperator{\Proj}{Proj}
\DeclareMathOperator{\II}{II}
\DeclareMathOperator{\smooth}{sm}
\DeclareMathOperator{\Sym}{Sym}

\DeclareMathOperator{\Hilb}{Hilb}
\DeclareMathOperator{\rank}{rank}

\DeclareMathOperator{\Hom}{Hom}
\DeclareMathOperator{\im}{Im}

\begin{document}

\title{On the hypersurfaces contained in their Hessian}


\author{Pietro De Poi}
\address{Dipartimento di Scienze Matematiche, Informatiche e Fisiche\\
Universit\`a degli Studi di Udine\\
via delle Scienze, 206\\
33100 Udine\\
Italy
}
\email{pietro.depoi@uniud.it}

\author{Giovanna Ilardi}
\address{Dipartimento di Matematica e Applicazioni
\\ Universit\`a degli Studi di Napoli ``Federico II''\\
80126 Napoli\\
Italy}
\email{giovanna.ilardi@unina.it}

\thanks{The authors
are members of INdAM-GNSAGA and the first author is supported by PRIN
``Geometria delle variet\`a algebriche'' and by Universit\`a degli Studi
di Udine -- DMIF project ``Geometry PRID ZUCC 2017''}



\date{\today}


\begin{abstract}
This article presents the theory of focal locus applied to the hypersurfaces in the projective space which 
are (finitely) covered by linear spaces and such that the tangent space is constant along these spaces. 
\end{abstract}

\maketitle

\section*{Introduction}

Let $X\subset \PP^n$ be an irreducible hypersurface given by a form $f$ of degree $d\ge 3$. Let  $\II_P$
 be the second fundamental 
form of $X$ at a general smooth point $P\in X$; it can be interpreted as a quadric $\II_P\subset \PP(T_PX)\cong \PP^{n-2}$. 
A smooth point $P$ is said to be \emph{$h$-parabolic} for $X$ if the quadric  $\II_P$ has rank $n-h-1$,  $1\le h \le n-1$, with 
the convention that the null quadric has rank $0$. 

B. Segre in \cite{Seb} proved that 
if every smooth point of $X$ is $h$-parabolic, then the polynomial $f^h$ divides 
the Hessian polynomial $h(f)$ of $X$, he  posed a  problem of finding the hypersurfaces $X$ with $h$-parabolic points such 
that $h(f)$ is divisible by $f^{h+1}$. 

Then the problem was studied by A. Franchetta in \cite{F2} and \cite{F1} where the $1$-parabolic 
case is considered in the case of, respectively, $\PP^3$ and $\PP^4$ and more  generally  by C. Ciliberto in \cite{Ciruzzo}. 
A. Franchetta used, to prove his results, the theory of focal locus  of families of lines introduced by C. Segre in \cite{CSeg}. This 
theory has not been carried on by other more recent researchers.   
\\ For the study of hypersurfaces with vanishing Hessian (i.e. such that $f^h$ divides the Hessian polynomial $h(f)$ of $X$, for every $h$) and their link with ideals failing the strong Lefschetz property see for example \cite{GoR}, \cite{Go}, \cite{GoZ}, \cite{Wa2}, \cite{Wa} and \cite{Wa1}.

 
In this paper, we follow the approach introduced in \cite{CS} for studying \emph{congruences of $k$-planes } $B$ of $\PP^n$, 
i.e. families of  $k$-planes
of dimension $n-k$ via their focal locus, giving a modern account of C. Segre's theory of \cite{Sf}. 
The number of $k$-planes passing through a  general point $P\in \PP^n$ is, by dimensional reasons, finite, and this number is 
called \emph{order} $a$ of the congruence. If this number is positive, i.e.,  the $k$-planes sweep out $\PP^n$, then the congruence 
is geometrically described by its \emph{focal locus} $F$, that is the branch locus of the map $f$ 
from the incidence correspondence to $\PP^n$: 
indeed, either a $k$-plane of $B$ is contained in $F$ or it intersects $F$ in hypersurface of degree $n-k$, the latter being the general 
case. If in particular $a=1$, i.e. $f$ is birational, by Zariski Main Theorem we have that $F$ coincides with the \emph{fundamental locus}, 
that is, the locus of points through which there passes infinitely many $k$-planes of $B$; for example, in the case $k=1$, $B$ is given 
by the $(n-1)$-secant lines to $F$ (see for example \cite{D1}, \cite{D2}, \cite{D3}, \cite{D4} and \cite{D5}). 

Unfortunately this geometric approach relies on the fact that---at least in the general case---the $k$-planes sweep out $\PP^n$, which 
is a smooth variety, and the fact that the normal space of a $k$-plane in $\PP^n$ is immediate to calculate. 

In this paper, we have adapted and applied the theory of focal locus
to the case of the hypersurfaces such that their smooth points are 
$h$-parabolic. Such a hypersurface $X\subset \PP^n$ can be characterized, see Proposition \ref{prop:gs}, as covered by a 
$\Sigma_X$ family $h$-planes, 
such that through a  general point $P\in \PP^n$ there passes only one $h$-plane of the family, see Remark \ref{rmk:111}. We have that  $\dim 
\Sigma_X=n-1-h$, and, in analogy of the case of $\PP^n$ we can call $\Sigma_X$ a \emph{congruence of $h$-planes of $X$ of order $1$}. 
Indeed, since $X$ is always singular, we have to take desingularizations of the varieties and maps considered. This can be done and 
it is contained in Section \ref{sec:foc}; the main problem here is that there is no  good description of the focal locus as in the 
case of $\PP^n$ and this is mainly due to the fact that  it is not easy to find the normal space of an $h$-plane in a desingularization 
of $X$. Though, something can be said: for example, we can give dimensional bounds on the focal locus $F$ of this family: as in the 
case of $\PP^n$, $F$ has codimension at least $2$ in (a desingularization of) $X$: see Theorem \ref{thm:n-3}. 
We note that, differently to the case of $\PP^n$, it 
can happen that $F=\emptyset$: we characterize these varieties in Proposition \ref{prop:con}. 
Finally, we show, using the language of the moving frames, that our definition of the focal locus is more restrictive than the classical 
one of C. Segre and A. Franchetta. 

\section{Preliminaries}

We shall adopt the usual notation and conventions of algebraic geometry as in \cite{H} and \cite{GH}.
We will work with schemes and varieties over the complex field $\CC$. If $V$ is a $\CC$-vector space
of dimension $n+1$, with fixed basis $(e_0,\dotsc,e_n)$ and with dual basis $(x_0,\dotsc,x_n)$ in the dual vector space
$V^*$, as usual $\PP^n=\PP(V)=\Proj\Sym V^*= \Proj(\CC[x_0,\dotsc,x_n])$ is the projective $n$-dimensional
space of $V$, and $G(2,V)$ is the Grassmannian of $2$-dimensional vector subspaces of $V$ (or, which is the same, of the
lines in $\PP(V)$). In what follows, we shall denote by $T_PX$ the Zariski tangent space and by
$\TT_PX\subset\PP^n$ the embedded
tangent space.

\section{$h$-parabolic points}

Let  us consider an irreducible hypersurface $X\subset \PP^n$ 
given by an irreducible form $f\in \CC[x_0,\dotsc,x_n]_d$ of degree $d\geq 3$.
Let $\II_P$ be the \emph{second fundamental form} of $X$ at a general smooth point $P\in X$.
Such a form can be seen, by definition, as a linear system generated by a quadric in $\PP(T_PX)\cong \PP^{n-2}$.
As such, $\II_P$
either consists of a unique quadric of rank $k$, where  $1\le k \le n-1$, or it is empty, in the case of the zero quadric.
Therefore, we can say that $\II_P$ is defined by a quadric of rank $k$, where $0\le k \le n-1$. By abuse of notation, sometimes we will 
identify this quadric with the second fundamental form  $\II_P$.

Obviously, we expect that the general point of the general hypersurface has $\II_P$ defined by a quadric of rank $n-1$.

Sometimes, it is convenient to consider the quadric defining the second fundamental form $\II_P$ as a quadric $Q(\II_P)$ in $\TT_PX$; 
clearly, this is the
set of lines passing through $P$ which intersect $X$ with multiplicity at least $3$.
\begin{definition}
Classically, $Q(\II_P)\subset \TT_PX$ is called
\emph{asymptotic cone}.  
\end{definition}
\begin{remark}
 The equation defining $Q(\II_P)$ is the same as  one defining $\II_P$, and the vertex of $Q(\II_P)$ is the span of $P$ and the
vertex of  $\II_P$.
\end{remark}

\begin{definition}\label{defHparabPoint}
 We call a smooth point $P$ \emph{$h$-parabolic} for $X$ if the quadric of $\II_P$ has rank $n-h-1$ ($1 \le h \le n-1$).
In other words, the dimension of the vertex of  $\II_P$ is $h-1$ (equivalently, the dimension of the vertex of $Q(\II_P)$ is $h$).
If $\II_P=\emptyset$, i.e. if $P$ is $(n-1)$-parabolic, we say also that $P$ is a \emph{flex}. The  
set of all $h$-parabolic points 
is the \emph{locus of the parabolic points} and its points are the \emph{parabolic points of $X$}. 
\end{definition}
\section{The Hessian}
Now, we recall the following basics.
\begin{definition}
The \emph{Hessian matrix} of $X=V(f)$ is the matrix formed by the second partial derivatives of $f$
\begin{equation*}
H(f):=\left(\frac{\partial^2f}{\partial x_i \partial x_j}\right)_{0\le i,j\le n}.
\end{equation*}
\end{definition}

\begin{definition}
The \emph{Hessian polynomial} of $X$ is the polynomial defined by
\begin{equation*}
h(f):=\det H(f).
\end{equation*}
\end{definition}
Clearly either $h(f)\equiv 0$ or $h(f)$ is a form of degree $(n+1)(d-2)$. In the latter case, the form defines a hypersurface,
which is called \emph{Hessian hypersurface} of $X$.

It is easy to see (see for example \cite[\S 5.6.4]{Bel}) the following 
\begin{proposition}
The locus of parabolic points (together with the singular locus)
is given by the intersection of the Hessian hypersurface with $X$ itself. 

Therefore the closure of the parabolic locus is the 
union of the parabolic locus with the singular locus of $X$.  
\end{proposition}

\begin{corollary}
If
all the (smooth) points of $X$ are  parabolic, then 
\begin{equation*}
X\subset V(h(f)), 
\end{equation*}
and vice versa. 
\end{corollary}

Indeed, one can prove  even more: 
\begin{theorem}
If  all the (smooth) points of an irreducible hypersurface $X\subset \PP^n$ defined by a form $f$ are  $h$-parabolic, then
\begin{equation*}
h(f)=f^h g,
\end{equation*}
that is $X$ is a \emph{$h$-tuple component of its Hessian}.
\end{theorem}
See \cite{Seb} and \cite[Corollary 4.4]{Z} for a modern account.

If, in particular, $h(f)\equiv 0$, then $X$ is
a $k$-tuple component   of its Hessian, for every $k\in \NN$, and $X$ is said to be, classically, \emph{with indeterminate Hessian}.


\section{The Gauss map}
Let us consider the Gauss map
\begin{align}
\gamma_1\colon X &{\dashrightarrow}\check{\PP^n}\label{eqGaussMap}\\
P &\mapsto  \TT_P X,\nonumber
\end{align}
where, as usual,  $\check{\PP^n}=\PP(V^*)$ is the dual of $\PP^n=\PP(V)$ and it is defined on the smooth locus of $X$, which is denoted by 
$X_{\smooth}$.
It is clear that in our case the (closure of the) image of $\gamma_1$ is the
(projective) \emph{dual variety of $X$}, which is denoted by $\check{X}\subset \check{\PP^n}$.

\begin{remark}\label{remGaussHParab}
 The differential of $\gamma_1$ at $P\in X_{\smooth}$, denoted by $d_P\gamma_1$, can be thought 
as the second fundamental form $\II_P$, see for example \cite[(1.62) p. 379]{G-H}.

Moreover, we recall that the (closure of the) fibres of the Gauss map are linear spaces, see
\cite[(2.10), p. 388]{G-H}. More precisely, if the rank of $d_P\gamma_1$ (or, which is the same, the
rank of $\II_P$) is $n-h-1(<n-1)$, or equivalently, $P$ is an $h$-parabolic point, then
\begin{equation*}
\PP^h_P:=\overline{\gamma_1^{-1}(\gamma_1(P))}\cong \PP^h.
\end{equation*}
Let us denote by $\sigma_P\in G(h+1,V)$ the point that corresponds to the $h$-dimensional subspace  $\PP^h_P\subset \PP^n$;
then the $\sigma_P$'s generate a subvariety
\begin{equation}\label{eq:sigma}
\Sigma_X:=\overline{\bigcup_{P\in X_{\smooth}} \sigma_P}\subseteq G(h+1,V).
\end{equation}

\end{remark}
After this remark, it is obvious that one can characterize the varieties with $h$-parabolic points by the following:

\begin{proposition}\label{prop:gs}
 If  all the smooth points of our irreducible hypersurface $X\subset\PP^n=\PP(V)$ defined by the form $f$ are $h$-parabolic,
the subvariety $\Sigma_X\subseteq G(h+1,V)$  of the fibres $\sigma_P$ of the Gauss map is such that
 \begin{enumerate}
 \item\label{gs:1} $\dim\Sigma_X=n-h-1$;
\item\label{gs:2} $X$ is covered by the elements of $\Sigma_X$;
\item\label{gs:3} if $P,Q_1,Q_2\in \PP^h_P$ are general points, then  $\TT_{Q_1} X= \TT_{Q_2} X= \TT_P X(=\gamma_1(P))$;
\item\label{gs:4} if $P\in X$ is a general point, then $\PP^h_P$ is the vertex of the asymptotic cone $Q(\II_P)$;
\item\label{gs:5} $X\subset V(h(f))$; more precisely $h(f)=f^h g$, i.e.  $X$ is a $h$-tuple component of its Hessian. 
\end{enumerate}
 Vice versa, if a hypersurface satisfies the properties \eqref{gs:1}, \eqref{gs:2} and either \eqref{gs:3} or \eqref{gs:4}, 
i.e. $X$ is covered by a family of 
$h$-planes of dimension $n-h-1$ such that in the smooth points of the general $h$-plane the tangent space to $X$ is constant, 
all its smooth points are $h$-parabolic. If $X$ satisfies \eqref{gs:5} and $h$ is maximal with this property, 
then there exists an $h'\le h$ such that all the smooth points of $X$ are $h'$-parabolic. 
\end{proposition}
\begin{proof}
Clearly, we have only to prove case \eqref{gs:4}; but this follows for example from the proofs of \cite[(2.6) and (2.10)]{G-H}: 
the vertex of $Q(\II_P)$ is defined by the forms in \cite[(2.8)]{G-H}, and letting these to move, 
they define the foliation on $X$ along which the tangent space remains constant and equal to  $\TT_P X$, 
i.e. the fibre of the Gauss map, and this is a linear space by \cite[(2.10)]{G-H}. 
\end{proof}

\section{Focal locus}\label{sec:foc}
Let us consider now a desingularization $S$ of our $(n-h-1)$-dimensional variety $\Sigma_X\subset G(h+1,V)$ defined in \eqref{eq:sigma}:
\begin{equation*}
s\colon S\to \Sigma_X\subset G(h+1,V);
\end{equation*}
associated to it we have the incidence correspondence
\begin{equation*}
I\subset S\times X\subset S\times \PP^n;
\end{equation*}
the projections induce maps
\begin{equation*}
\xymatrix{&\ar[dl]_p\ar[dr]^qI&\\
S&&X}
\end{equation*}
and $p\colon I\to S$ gives a flat family of $h$-planes of $\PP^n$;
but these $h$-planes are contained in $X$, which is---in general---a singular variety.
We can consider the following rational map
\begin{align*}
\phi\colon X&\dashrightarrow S\\
P&\mapsto s(\sigma_P)
\end{align*}
which associates to the (smooth) point $P\in X_{\smooth}$ its fibre of the Gauss map.
Let us consider a desingularization $\xi\colon \Xi \to X$ of $X$ which resolves also the indeterminacy of $\phi$:
\begin{equation*}
\xymatrix{&\ar[dl]_{\varphi}\ar[dr]^\xi\Xi&\\
S\ar@{<--}[rr]^\phi &&X}
\end{equation*}
and the corresponding fibre product $\Lambda:=I\times_X \Xi$:
\begin{equation*}
\xymatrix{S&I\ar[l]^p \ar[d]_q&\Lambda \ar[l]^\epsilon\ar[d]^\eta\ar@/_/[ll]_\pi\\
&X& \Xi\ar[l]_\xi}.
\end{equation*}
Since $\epsilon $ and $p$ are flat morphisms, the composition morphism $\pi\colon \Lambda\to S$ is flat too.
We recall the following:
\begin{definition}
A \emph{congruence of $k$-planes in $\PP^n$} is a flat family $(\Lambda, B, p )$ of $k$-planes of
$\PP^n$ obtained by a desingularization of a subvariety $B'$ of dimension $n-k$ of
the Grassmannian $G (k+1, n+1 )$ of $k$-planes of $\PP^n$. Notice that  $p$ is the restriction of the projection
$p_1\colon B \times \PP^n\to B$ to $\Lambda$. By dimensional reasons, there passes a finite number of elements this family 
though the general point of $\PP^n$, and this number is called the \emph{order} of the congruence.

Therefore, in analogy with the case of $\PP^n$, we can call the above flat
family $(\Lambda,S,\pi)$ as a \emph{congruence of $h$-planes in $X$}. In this case $\dim X=n-1$ and $\dim S=n-h-1$, therefore
there passes a finite number of elements the family though the general point of $X$.
We will observe in Remark \ref{rmk:111} that this number is one.
\end{definition}

Therefore, we can define the above triple $(\Lambda,B,\pi)$ as a \emph{congruence of $h$-planes in $X$}.

We note that $\pi$ gives a structure of a \emph{$\PP^h$-bundle} on $S$ to $\Lambda$.

Now we observe that---from the natural inclusion of the product $ \Lambda\to S\times \Xi$---we can construct
the \emph{focal diagram} associated to this flat family of $h$-planes:
\begin{equation*}
\xymatrix{
&&0\ar[d]&&\\
&&\Tf_{S\times \Xi/ \Xi}\mid_\Lambda \ar[d]\ar[rd]^{\lambda}&&\\
0\ar[r]&\Tf_\Lambda\ar[r]\ar[dr]_{d\eta}&\Tf_{S\times \Xi}\mid_\Lambda \ar[r]\ar[d]&\Nf_{\Lambda / S\times \Xi}\ar[r]& 0\\
&&\eta^*\Tf_{\Xi} \ar[d]&&\\
&&0&&
}
\end{equation*}
where $\Tf_{S\times \Xi/ \Xi}=\Ho(\Omega^1_{(S\times \Xi)/\Xi},\OO_{S\times \Xi})$ is the relative tangent sheaf of $S\times \Xi$
with respect to $\Xi$ (and observe that, if $p_1\colon S\times \Xi\to S $ is the projection, $p^*_1 \Tf_I\cong\Tf_{S\times \Xi/ \Xi}$) and
$\lambda$ is given by composition and it is called \emph{global characteristic map} of our family.
If we restrict $\lambda$ to a fibre $\Lambda_s:=\pi^{-1}(s)$, $s\in S$ of the family $\Lambda$, we obtain the \emph{characteristic map of
the family relative to $s$}:
\begin{equation*}
\lambda(s)\colon T_{s}S\otimes \OO_{\Lambda(s)}\to \Nf_{\Lambda(s) / \Xi}
\end{equation*}
where $\Lambda(s)\subset \Xi$ is the image of $\Lambda_s$ in $\Xi$. Passing to global sections, we get a linear map
\begin{equation*}
\ell(s)\colon T_{s}S\to H^0(\Lambda(s),\Nf_{\Lambda(s) / \Xi})
\end{equation*}
which is the differential of the functorial map $S\to \Hilb^{q(x)}_\Xi$.

Now, the condition
\begin{equation*}
\rank \lambda < n-h-1=\rank \Tf_{S\times \Xi/ \Xi}\mid_\Lambda = \rank \Nf_{\Lambda / S\times \Xi}
\end{equation*}
defines a closed subscheme of $\Lambda$, and from the focal diagram clearly we have
\begin{equation*}
\ker \lambda \cong \ker d\eta
\end{equation*}
and this closed subscheme gives the ramification divisor of $\eta$.
\begin{definition}
The image  $F\subset \Xi$  of $\eta$, 
i.e. the branch locus of $\eta$ is called the
\emph{focal locus} of the family $\Lambda$ of $h$-planes contained in $X$. Its points are called \emph{focal points}.
\end{definition}
\begin{remark}\label{rmk:111}
Now, it is immediate to observe that the map $\eta$ is birational, since through the general point of $X$ (and therefore of $\Xi$) there
passes only one $h$-space of the family: it is said that the family has \emph{order one}.

Therefore, by Zariski Main  Theorem, we deduce that the focal
locus coincides with the \emph{fundamental locus}, that is, the locus of the points of $\Xi$ for which the fibre under $\eta$ has positive
dimension.
\end{remark}
From this we deduce that
\begin{theorem}\label{thm:n-3} If $X\subset \PP^n$ is a hypersurface whose smooth points are $h$-parabolic and $F$ is the focal locus of the family
of $h$-planes contained in $X$, then
\begin{equation*}
\dim F \le n-3,
\end{equation*}
i.e. $F$ has codimension at least $2$ in $\Xi$.
\end{theorem}
Therefore we have
\begin{corollary}
If $X\subset \PP^n$ is a hypersurface that is a component of its Hessian
then
\begin{equation*}
\dim F \le n-3.
\end{equation*}
In particular,   this is true for a hypersurface with indeterminate Hessian.
\end{corollary}

It can indeed happen that $F=\emptyset$:
\begin{proposition}\label{prop:con}
Under the assumptions  from the previous theorem,
$F=\emptyset$ if and only if the $h$-planes of $X$ give a structure of $\PP^h$-bundle on $\Xi$.
\end{proposition}

\begin{proof}
It follows from the fact that  $F=\emptyset$ if and only if the map $\eta$ is an isomorphism, and that $\Lambda$ is a $\PP^h$-bundle over
$S$.

\end{proof}

\begin{remark}
Since $\eta\colon\Lambda\to \Xi $ is birational and $\Lambda$ and $\Xi$ are smooth, we can think of $\eta$ as the blow-up of $\Xi$ along 
the focal locus $F$.  
\end{remark}

We recall that (Proposition \ref{prop:gs}, \eqref{gs:3}) that the tangent hyperplane is constant along the points of a 
fixed $h$-plane of $X$.

\begin{example}\label{ex:1}
The simplest case in which Proposition \ref{prop:con} applies is clearly the case in which $X$ is a cone: in this case a desingularization of
$X$ is obviously a $\PP^h$-bundle. 
\end{example}

\begin{example}\label{ex:2}
Let $Y\subset\PP^{2h+1}$ be a \emph{non-degenerate} (i.e.  not contained in a hyperplane) \emph{non-defective}
(i.e. such that its secant variety has expected dimension) variety of dimension $h$;
then, its \emph{variety of tangents} $\TT Y:=\overline{\cup_{P\in Y_{\smooth}}\TT_P Y} \subset \PP^{2h+1}$ is a \emph{developable} hypersurface, whose
desingularization is clearly the projectivized tangent bundle of $Y$.  
\end{example}

 Asking  if there are more examples beyond Examples \ref{ex:1} and \ref{ex:2},
 we need to focus on the infinitesimal behaviour of our variety, which will be done in Section \ref{sec:mf}.

\section{Moving frames}\label{sec:mf}

To study the behaviour of $X$ in $P\in X_{\smooth}$,
following \cite{G-H} (and references [2], [6], [7] and [10] therein), and using notation as in
\cite{ddi} and \cite{di}
we consider the manifold $\mathcal F (X)$ of frames in $X$.
 An element of  $\mathcal F (X)$ is a \emph{Darboux frame} centred in $P$.
This means an $(n+1)$-tuple
\begin{equation*}
\left\{A_0;A_1,\dotsc,A_{n-1};A_n\right\}
\end{equation*}
which is a basis of $V\cong \CC^{n+1}$ such that, if $\pi:V\setminus\{0\}\rightarrow \PP^N$ is the canonical projection,
\begin{align*}
&\pi(A_0)=P,&
&\textup{ and }&
&\pi(A_0), \pi(A_1),\dotsc,\pi(A_{n-1})\text{ span } \TT_P(X).
\end{align*}
Here and in this section, following \cite{G-H}, by abusing notation, we identify the embedded tangent space with its affine cone.

Let this frame move in  $\mathcal F (X)$;
then we have the following structure equations (in terms of the restrictions to $X$ of
the Maurer-Cartan $1$-forms $\omega_i$, $\omega_{i,j}$ on $\mathcal F(\PP^n)$)
for the exterior derivatives of this moving frame

\begin{equation}\label{derivate_frame_eq}
\begin{cases}
 \omega_n=0 &\\
  d A_0=\sum_{i=0}^{n-1} \omega_i A_i \\
  d A_i=\sum_{j=0}^n \omega_{i,j}A_j  &i=1,\dotsc,N\\
  d\omega_j= \sum_{h=0}^{n-1} \omega_h \wedge \omega_{h,j} & j=0,\dotsc,n-1\\
    d\omega_{i,j}= \sum_{h=0}^n \omega_{i,h} \wedge \omega_{h,j}  &i=1,\dotsc,n,\ j=0,\dotsc,n.
\end{cases}
\end{equation}
Following \cite{G-H}, we have that, in this notation, the second fundamental form of $X$ in $P$
 is given by the quadric
\begin{equation}\label{eq:secondo}
V^{(2)}=\sum_{i,j=1}^{n-1}\omega_i\omega_{i,n} =\sum_{i,j=1}^{n-1}q_{i,j}\omega_{i} \omega_{j} ,
\end{equation}
where $q_{i,j}(=q_{j,i})$ are defined by
\begin{equation}\label{eq:secondobis}
\omega_{i, n}=\sum_{j}q_{i,j} \omega_{j},
\end{equation}
which are obtained via the Cartan Lemma from
\begin{equation*}
0=d\omega_{n}=\sum_{i}\omega_{i}\wedge \omega_{i,n},
\end{equation*}
since $\omega_{n}=0$ on $\TT_P(X)$.

 The (projective) Gauss map can be expressed, using the above Darboux frame, as the rational map
\begin{align*}
    \gamma_1\colon &X \dashrightarrow \check{\PP^n}\\
    & P \mapsto A_0\wedge\dotsb \wedge A_{n-1},
\end{align*}
and therefore by \eqref{derivate_frame_eq},
\begin{align*}
 d\gamma_1&\equiv\sum_{
i=1}^{n-1}(-1)^{n-i} \omega _{i,n} A_0\wedge\dotsb\wedge \hat {A_i}\wedge\dotsb \wedge A_{n-1}\wedge A_{n}, & \mod \TT_P X;
\end{align*}
from which one can deduce that
$d\gamma_1$ at $P$ can be interpreted as the second fundamental form at $P$,
since, thanks to the canonical isomorphism $\bigwedge^{n-1}V\cong V^*$, we have
\begin{equation*}
\omega_i\cong (-1)^{n-i} A_0\wedge\dotsb\wedge \hat {A_i}\wedge\dotsb \wedge A_{n-1}\wedge A_{n},
\end{equation*}
see \eqref{eq:secondobis}.

Let now suppose as above that the Gauss map has fibres $\PP^h_P$ of dimension $h$; this happens---as we have seen---if
and only if $\rank (d\gamma_1)=n-h-1$ and if and only if
the space
\begin{equation*}
U^*:=\gen{\omega _{i,n}}_{i=1,\dotsc, n-1}\subset T^*_P(X)
\end{equation*}
has dimension $n-h-1$. Dually, this space defines a subspace  $U\subset T_P(X)$ of dimension $h$, defined by the equations
\begin{align*}
\omega _{i,n}&=0, & i=1,\dotsc, n-1.
\end{align*}
We choose a frame such that $\omega_{h+1},\dotsc,\omega_{n-1}$ form a basis for $U^*$,
that is
\begin{equation}\label{eq:genero}
\gen{\omega _{i,n}}_{
i=1,\dotsc, n-1}=\gen{\omega_{h+1},\dotsc,\omega_{n-1}}.
\end{equation}
Indeed, the quadric of the second fundamental form is a cone with vertex  $\PP^{h-1}=\PP(U)$, see \cite[(2.6)]{G-H} and has equation,
see \cite[(3.11)]{di}
\begin{equation*}
V^{(2)}=\sum_{k_1,k_2=h+1}^{n-1}q_{k_1, k_2}\omega_{k_1}\omega_{k_2}.
\end{equation*}
Of course, the cone over $\PP^{h-1}=\PP(U)$ with vertex $P$ in $\TT_PX$
is the (closure of the) fibre of the Gauss map $\PP^h_P\cong \PP^h$, see \cite[(2.10)]{G-H} and Proposition
\ref{prop:gs}, \eqref{gs:4}.

\subsection{Focal locus and moving frames}
We start by recalling the notation introduced in \cite[\S 2(a)]{G-H}.  Let $B$ be an $r$-dimensional variety,  
\begin{equation*}
f\colon B \to G(h+1,V)
\end{equation*}
be a morphism, $y\in B$ be a general point, and let $S_y\subset V$ (where as above $\PP^n=\PP(V)$)
be the $(h+1)$-dimensional vector space which
corresponds to $f(y)$. Then, the differential of $f$ in $y$ can be thought of
\begin{equation*}
    d_y f \colon T_y B \to  \Hom(S_y, N_y),
\end{equation*}
where $N_y:=\frac{V}{S_y}$. More explicitly, if $y_1,\dotsc, y_r$ are local coordinates of $B$ near $y$, and if $e_0(y), \dotsc e_h(y)$ is
a basis for the  $(h+1)$-dimensional vector spaces near  $S_y$, we have
\begin{equation*}
    d_y f\left(\frac{\partial }{\partial y_i}\right)(e_j(y))\equiv \frac{\partial (e_j(y))}{\partial y_i}\mod S_y
\end{equation*}
for $i=1,\dotsc,r$ and $j=0,\dotsc h$. Then fixing $w\in  T_y B$, one can define the \emph{infinitely near space to $S_y$ in
the direction of $w$} as
\begin{equation*}
\frac{S_y}{dw}:=\im d_y f(w)\subset N_y,
\end{equation*}
and---if we denote by $[v]$ the class of $v\in V$ in $N_y$---we consider the following subspaces of $V$:
\begin{align*}
S_y+\frac{S_y}{dw}&:=\left\{v\in V\mid v=s+t, s\in S_y, [t]\in \frac{S_y}{dw}\right\}\\
S_y\cap\frac{S_y}{dw}&:=\left\{s\in S_y\mid [s]\in  \frac{S_y}{dw}\right\}=\ker (d_y f(w)).
\end{align*}
Obviously, we have
\begin{align}\label{eq:ro}
\rank(d_yf(w))&=\rho& &\iff &\dim(S_y+\frac{S_y}{dw})&=h+1+\rho& &\iff &\dim(S_y\cap\frac{S_y}{dw})&=h+1-\rho.
\end{align}
We also denote the projective subspaces of $\PP^n$ associated to the vector spaces just defined by
\begin{align*}
\PP_y^h+\frac{\PP_y^h}{dw}&:=\PP(S_y+\frac{S_y}{dw})\\
\PP_y^h\cap\frac{\PP_y^h}{dw}&:=\PP(S_y\cap\frac{S_y}{dw}).
\end{align*}

We apply now this to the situation of Section \ref{sec:foc}  with $S=B$, $s=f$,  and $r=n-h-1$. Moreover, $\PP_y^h=\PP^h_P$ for
some $P\in X$ (we can suppose that $P$ is a smooth point), with
$y\in S$.
A count with coordinates shows, see \cite[(2.18)]{G-H}
\begin{equation}\label{eq:2}
\PP_y^h+\frac{\PP_y^h}{dw}=\PP^h_P+\frac{\PP^h_P}{dw} \subset\TT_PX=: \PP^{n-1}_y.
\end{equation}
The \emph{focal points} in a $\PP^h_P$,  as defined classically by C. Segre and A. Franchetta,  
are given by $\PP^h_P\cap\frac{\PP^h_P}{dw}$ and the points in
this set are contained in the singular locus of $X$
(\cite[(page 393)]{G-H}). 
\begin{remark}
Indeed, our definition of focal point is more restrictive than the classical one: for example, 
the points of the vertex of a cone (see Example \ref{ex:1}) or the points of the variety $Y$  from Example \ref{ex:2} are focal points 
in the classical sense, but these are not focal points for us. 
\end{remark}
This fact can also be seen in the following way:  
if we do not have the focal locus in the classical sense, we must have that
\begin{equation*}
\PP^h_P\cap\frac{\PP^h_P}{dw}= \emptyset 
\end{equation*}
which is equivalent the requirement that 
\begin{equation*}
\dim \PP^h_P+\frac{\PP^h_P}{dw}= 2h+1 
\end{equation*}
and therefore $h\le \frac{n}{2}-1$.

\bibliographystyle{amsalpha}
\bibliography{dim6}

\end{document}